\documentclass[12pt]{amsart}

\usepackage{amssymb, amsmath, amscd, newlfont, bbm, slashbox}
\usepackage[foot]{amsaddr}

\usepackage{hyperref} 
\hypersetup{
	colorlinks=true,        
	linkcolor={black},  
	citecolor={black},
	urlcolor={black}     
}
\usepackage{enumitem}
\usepackage{tikz-cd}
\usepackage{comment}
\usepackage{cite, mathtools}

\newcommand{\spacedcdot}{{\,\cdot\,}}

\newcommand{\Q}{{\mathbb{Q}}}
\newcommand{\C}{{\mathbb{C}}}
\newcommand{\R}{{\mathbb{R}}}
\newcommand{\Z}{{\mathbb{Z}}}

\newcommand{\diag}{{\mathrm{diag}}}

\newcommand{\GL}{{\mathrm{GL}}}

\newcommand{\Sp}{{\mathrm{Sp}}}

\newcommand{\SL}{{\mathrm{SL}}}
\newcommand{\U}{{\mathrm{U}}}

\newcommand{\calH}{{\mathcal{H}}}

\usepackage[mathscr]{eucal}

\providecommand{\abs}[1]{\left\lvert#1\right\rvert}
\providecommand{\norm}[1]{\left\lVert#1\right\rVert}
\providecommand{\scal}[2]{\left<#1,#2\right>}

\newtheorem{theorem}{Theorem}[section]
\newtheorem{lemma}[theorem]{Lemma}
\newtheorem{proposition}[theorem]{Proposition}
\newtheorem{corollary}[theorem]{Corollary}

\theoremstyle{definition}

\theoremstyle{remark}
\newtheorem{remark}[theorem]{Remark}
\newtheorem*{remark*}{Remark}

\newtheoremstyle{named}{}{}{\itshape}{}{\bfseries}{.}{.5em}{#1 \thmnote{#3}}
\theoremstyle{named}

\numberwithin{equation}{section}

\title[]{Orthogonality relations for Poincar\'e series}

\author{Sonja \v Zunar}

\address{University of Zagreb,
	Faculty of Geodesy,
	Ka\v ci\'ceva 26,
	10000 Zagreb,
	Croatia}
\email{szunar@geof.hr}

\thanks{This work is supported by the Croatian Science Foundation under the project number HRZZ-IP-2022-10-4615.}

\begin{document}

\begin{abstract}
	Let $ G $ be a connected semisimple Lie group with finite center. We prove a formula for the inner product of two cuspidal automorphic forms on $ G $ that are given by Poincar\'e series of $ K $-finite matrix coefficients of an integrable discrete series representation of $ G $. As an application, we give a new proof of a well-known result on the Petersson inner product of certain vector-valued Siegel cusp forms. In this way, we extend results previously obtained by G.\ Mui\'c for cusp forms on the upper half-plane, i.e., in the case when $ G=\SL_2(\R) $.
\end{abstract}

\maketitle

\section{Introduction}

In \cite{muic12}, G.\ Mui\'c gave a representation-theoretic proof of a formula for the Petersson inner product of two cusp forms in $ S_m(\Gamma) $, where $ m\in\Z_{\geq3} $, $ \Gamma $ is a discrete subgroup of finite covolume in $ \SL_2(\R) $, and one of the cusp forms is given by a Poincar\'e series of polynomial type that lifts in a standard way to a Poincar\'e series of a $ K $-finite matrix coefficient of an integrable discrete series representation of $ \SL_2(\R) $. Let $ G $ be a connected semisimple Lie group with finite center. In this paper, by generalizing the methods of \cite{muic12}, we prove a formula for the inner product of two cuspidal automorphic forms on $ G $ that are given by Poincar\'e series of $ K $-finite matrix coefficients of an integrable discrete series representation of $ G $. As an application, we give a new proof of a well-known result on the Petersson inner product of certain vector-valued Siegel cusp forms. 

To explain our results in more detail, let us fix a maximal compact subgroup $ K $ of $ G $ and a Haar measure $ dg $ on $ G $. Let $ \mathfrak g $ denote the Lie algebra of $ G $, and let $ \mathcal Z(\mathfrak g) $ denote the center of the universal enveloping algebra $ \mathcal U(\mathfrak g) $ of the complex Lie algebra $ \mathfrak g_\C=\mathfrak g\otimes_\R\C $. Let $ (\pi,H) $ be a discrete series representation of $ G $, and let $ H_K $ denote the $ (\mathfrak g,K) $-module of $ K $-finite vectors in $ H $.

It is well known that for all $ h,h'\in H $, the matrix coefficient $ c_{h,h'}:G\to\C $,
\begin{equation}\label{eq:003}
	c_{h,h'}(g)=\scal{\pi(g)h}{h'}_H, 
\end{equation}
belongs to $ L^2(G) $ \cite[\S4.5.9]{warner72}. Moreover, by the Schur orthogonality relations \cite[Theorem 4.5.9.3]{warner72}, we have
\begin{equation}\label{eq:005}
	\scal{c_{h,h'}}{c_{v,v'}}_{L^2(G)}=\frac1{d(\pi)}\scal{h}{v}_H\scal{v'}{h'}_H,\qquad h,h',v,v'\in H,  
\end{equation}
where $ d(\pi)\in\R_{>0} $ is the formal degree of $ \pi $, i.e., the Plancherel measure of $ \pi $ \cite[Proposition 18.8.5]{dixmier77}. On the other hand, if $ (\sigma,V) $ is a discrete series representation of $ G $ that is not equivalent to $ \pi $, then we have
\[ \scal{c_{h,h'}}{c_{v,v'}}_{L^2(G)}=0,\qquad h,h'\in H,\ v,v'\in V.  \]

Next, suppose that $ (\pi,H) $ belongs to the integrable discrete series of $ G $, i.e., we have $ c_{h,h'}\in L^1(G) $ for all $ h,h'\in H_K $. Let $ \Gamma $ be a discrete subgroup of $ G $. Then, given $ h,h'\in H_K $, the Poincar\'e series 
\[ \left(P_\Gamma c_{h,h'}\right)(g)=\sum_{\gamma\in\Gamma}c_{h,h'}(\gamma g) \]
converges absolutely and uniformly on compact subsets of $ G $ and defines a smooth, left $ \Gamma $-invariant, $ \mathcal Z(\mathfrak g) $-finite, right $ K $-finite function that belongs to $ L^p(\Gamma\backslash G) $ for every $ p\in\left[1,\infty\right] $ (see, e.g., \cite[Lemma 3.2]{zunar25b}). If $ G=\mathcal G(\R) $ for some connected semisimple algebraic group $\mathcal G $ defined over $ \Q $, and $ \Gamma $ is an arithmetic subgroup of $ \mathcal G(\Q) $, then the function $ P_\Gamma c_{h,h'} $ is a cuspidal automorphic form for $ \Gamma $ \cite[Lemma 4-2(i)]{muic16}.
D.\ Mili\v ci\'c proved that
\[ \mathrm{span}_\C\left\{P_\Gamma c_{h,h'}:h,h'\in H_K\right\} \]
is dense in the $ \pi $-isotypic component $ L^2(\Gamma\backslash G)_{[\pi]} $ of the right regular representation of $ G $ on $ L^2(\Gamma\backslash G) $ \cite[Lemma 6.6]{muic19}.

The main result of this paper is the following generalization of Schur orthogonality relations.

\begin{theorem}\label{thm:001}
	Let $ G $ be a connected semisimple Lie group with finite center and a fixed Haar measure $ dg $, and let $ \Gamma $ be a discrete subgroup of $ G $. Let $ (\pi,H) $ be an integrable discrete series representation of $ G $. Then, we have the following:
	\begin{enumerate}[label=\textup{(\roman*)}]
		\item\label{thm:001:1} \textbf{\textup{(Inner product formula)}} For all $ h,h',v,v'\in H_K $, we have 
		\begin{equation}\label{eq:008}
			\scal{P_\Gamma c_{h,h'}}{P_\Gamma c_{v,v'}}_{L^2(\Gamma\backslash G)}=\frac1{d(\pi)}\scal{h}{v}_H\left(P_{\Gamma}c_{v',h'}\right)(1_G), 
		\end{equation}
		where $ 1_G $ denotes the identity element of $ G $, and $ d(\pi) $ denotes the formal degree of $ \pi $.
		\item\label{thm:001:2} Let $ (\sigma,V) $ be an integrable discrete series representation of $ G $ that is not equivalent to $ \pi $. Then, for all $ h,h'\in H_K $ and $ v,v'\in V_K $, we have
		\[ \scal{P_\Gamma c_{h,h'}}{P_\Gamma c_{v,v'}}_{L^2(\Gamma\backslash G)}=0. \]
	\end{enumerate}
\end{theorem}

Our proof of Theorem \ref{thm:001} is inspired by \cite[proof of Theorem 2.3]{muic12} and relies on the results on Poincar\'e series $ P_\Gamma c_{h,h'} $ proved in \cite{muic19}.

As an application of the inner product formula \eqref{eq:008}, in Section \ref{sec:110} we give a representation-theoretic proof of a well-known formula for the Petersson inner product of certain vector-valued Siegel cusp forms (see, e.g., \cite[Lemma 1.2]{zunar25a}). To explain this in more detail, let $ n\in\Z_{>0} $, and let $ (\rho,V) $ be an irreducible polynomial representation of $ \GL_n(\C) $ of highest weight $ \omega=(\omega_1,\ldots,\omega_n)\in\Z^n $, where $ \omega_1\geq\ldots\geq\omega_n>2n $. 
We equip $ V $ with the Hermitian inner product $ \scal\spacedcdot\spacedcdot_V $ with respect to which the restriction $ \rho\big|_{\U(n)} $ is unitary. Let $ \Gamma $ be a discrete subgroup of $ \Sp_{2n}(\R) $ that is commensurable with $ \Sp_{2n}(\Z) $.

A Siegel cusp form of weight $ \rho $ for $ \Gamma $ is a $ V $-valued holomorphic function $ \phi $ on the Siegel upper half-space
\[ \calH_n=\left\{z=x+iy\in M_n(\C):z^\top=z\text{ and }y>0\right\} \]
with the following two properties:
\begin{enumerate}
	\item $ \phi((Az+B)(Cz+D)^{-1})=\rho(Cz+D)\,\phi(z) $ for all $ \gamma=\begin{pmatrix}A&B\\C&D\end{pmatrix}\in\Gamma $ and $ z\in\calH_n $.
	\item $ \sup_{z\in\calH_n}\norm{\rho\left(y^{\frac12}\right)\phi(z)}_V<\infty $.
\end{enumerate}
Siegel cusp forms of weight $ \rho $ for $ \Gamma $ constitute a finite-dimensional complex vector space $ S_\rho(\Gamma) $, which we equip with the Petersson inner product
\[ \scal{\phi_1}{\phi_2}_{S_\rho(\Gamma)}=\frac1{\abs{\Gamma\cap\left\{\pm I_{2n}\right\}}}\int_{\Gamma\backslash\calH_n}\scal{\rho(y)\phi_1(z)}{\phi_2(z)}\,\det y^{-n-1}\hspace{-1.5mm}\prod_{1\leq r\leq s\leq n}dx_{r,s}\,dy_{r,s}. \]
It is well known (see, e.g., \cite[Theorem 1.1]{zunar25a}) that the space $ S_\rho(\Gamma) $ is spanned by the absolutely and locally uniformly convergent Poincar\'e series
\[ \begin{aligned}
	\left(P_{\Gamma,\rho}f_{\mu,v}\right)(z)=\sum_{\left(\begin{smallmatrix}A&B\\C&D\end{smallmatrix}\right)\in\Gamma}&\mu\bigg(\Big(Az+B-i(Cz+D)\Big)\Big(Az+B+i(Cz+D)\Big)^{-1}\bigg)\\
	&\cdot\rho\left(\frac1{2i}\Big(Az+B+i(Cz+D)\Big)\right)^{-1}v,
\end{aligned} \]
where $ v $ goes over $ V $ and $ \mu $ goes over the ring $ \C[X_{r,s}:1\leq r,s\leq n] $ of polynomials with complex coefficients in the $ n^2 $ variables $ X_{r,s} $ with $ r,s\in\{1,\ldots,n\} $. In \cite{zunar25a}, we showed that these Poincar\'e series lift in a standard way to Poincar\'e series of $ K $-finite matrix coefficients of an integrable antiholomorphic discrete series representation $ \pi_\rho^* $ of $ \Sp_{2n}(\R) $. By using a variant of this lift to transfer the inner product formula \eqref{eq:008} from $ L^2(\Gamma\backslash\Sp_{2n}(\R)) $ to $ S_\rho(\Gamma) $, we obtain the following well-known result that extends \cite[Corollary 1.2]{muic12} and \cite[Corollary 1.4]{zunar23} and illuminates the connection of the Siegel cusp forms $ P_{\Gamma,\rho}f_{1,v} $ with the reproducing kernel function for $ S_\rho(\Gamma) $ (cf.\ \cite{godement57}), as described in detail in \cite[Lemma 1.2 and \S7]{zunar25a}.

\begin{corollary}\label{cor:111}
	Let $ \phi\in S_\rho(\Gamma) $ and $ v\in V $. Then, we have
	\[ \scal{\phi}{P_{\Gamma,\rho}f_{1,v}}_{S_\rho(\Gamma)}=\frac{\dim_\C V}{d(\pi_\rho)}\scal{\phi(iI_n)}v_{V}. \]
\end{corollary}

\section{Proof of the inner product formula}

Let $ G $ be a connected semisimple Lie group with finite center, and let $ \Gamma $ be a discrete subgroup of $ G $. 
We fix a Haar measure $ dg $ on $ G $, which induces the $ G $-invariant Radon measure on $ \Gamma\backslash G $ such that
\begin{equation}\label{eq:105}
	\int_{\Gamma\backslash G}\sum_{\gamma\in\Gamma}f(\gamma g)\,dg=\int_G f(g)\,dg 
\end{equation}
for every compactly supported, continuous function $ f:G\to\C $
\cite[Theorem 8.36]{knapp02}.
Given a finite-dimensional complex Hilbert space $ V $ and $ p\in\left[1,\infty\right] $, let $ L^p(\Gamma\backslash G,V) $ denote the $ L^p $ space of functions $ \Gamma\backslash G\to V $. We write $ L^p(\Gamma\backslash G)=L^p(\Gamma\backslash G,\C) $.
Given an irreducible unitary representation $ \pi $ of $ G $, let $ L^2(\Gamma\backslash G)_{[\pi]} $ denote the $ \pi $-isotypic component of the right regular representation $ \left(R,L^2(\Gamma\backslash G)\right) $, i.e., the closure of the sum of closed irreducible $ G $-invariant subspaces of $ L^2(\Gamma\backslash G) $ that are equivalent to $ \pi $.

Given a unitary representation $ (\pi,H) $ of $ G $ and a function $ f\in L^1(G) $, a continuous linear operator $ \pi(f):H\to H $ is standardly defined by the condition
\[ \scal{\pi(f)h}{h'}_H=\int_Gf(g)\scal{\pi(g)h}{h'}_H\,dg,\qquad h,h'\in H, \]
where $ \scal\spacedcdot\spacedcdot_H $ denotes the inner product on $ H $ \cite[(1.8b)]{knapp86}.
It is well known that for all $ f\in L^1(G) $ and $ \varphi\in L^2(\Gamma\backslash G) $, we have
\begin{equation}\label{eq:002}
	\left(R(f)\varphi\right)(g)=\int_G f(g')\,\varphi(gg')\,dg'\qquad\text{for a.a.\ }g\in G. 
\end{equation}
Moreover, by the dominated convergence theorem, we have the following lemma.

\begin{lemma}\label{lem:006}
	Let $ f\in L^1(G) $, and let $ \varphi:\Gamma\backslash G\to\C $ be a bounded, continuous function that belongs to $ L^2(\Gamma\backslash G) $. Then, the right-hand side of \eqref{eq:002} defines a continuous function $ G\to\C $.
\end{lemma}

	In what follows, we use the notation $ c_{h,h'} $ (resp., $ d(\pi) $) introduced in \eqref{eq:003} (resp., \eqref{eq:005}) for a matrix coefficient (resp., the formal degree) of a discrete series representation $ \pi $ of $ G $. 
	The following lemma is an immediate corollary of \eqref{eq:005}.

	\begin{lemma}\label{lem:004}
	Let $ (\pi,H) $ be a discrete series representation of $ G $. Let $ h'\in H\setminus\left\{0\right\} $. Then, the assignment 
	\[ h\mapsto \frac{d(\pi)^{\frac12}}{\norm{h'}_H}c_{h,h'} \]
	defines a unitary $ G $-equivalence $ \Phi_{h'} $ from $ H $ to the irreducible closed $ G $-invariant subspace
	\[ c_{H,h'}:=\left\{c_{h,h'}:h\in H\right\} \]
	of the right regular representation $ \left(R,L^2(G)\right) $.
	\end{lemma}

Let $ \mathfrak g $ denote the Lie algebra of $ G $, and let $ K $ be a maximal compact subgroup of $ G $. Given a unitary representation $ (\pi,H) $ of $ G $, we denote by $ H_K $ the $ (\mathfrak g,K) $-module of $ K $-finite vectors in $ H $.

\begin{lemma}\label{lem:005}
	Let $ (\pi,H) $ be an integrable discrete series representation of $ G $. Let $ h,h'\in H $ and $ v,v'\in H_K $. Then, we have
	\begin{equation}\label{eq:006}
		R(\overline{c_{v,v'}})c_{h,h'}=\frac1{d(\pi)}\scal{h}{v}_Hc_{v',h'}. 
	\end{equation}
\end{lemma}		

\begin{proof}
	By \eqref{eq:002} and Lemma \ref{lem:006}, the continuous representative of the equivalence class $ R\left(\overline{c_{v,v'}}\right)c_{h,h'}\in L^2(G) $ is given by
	\begin{align*}
		\left(R(\overline{c_{v,v'}})c_{h,h'}\right)(g)
		&\overset{\phantom{\eqref{eq:005}}}=\int_G\overline{c_{v,v'}(g')}\,c_{h,h'}(gg')\,dg'\\
		&\overset{\phantom{\eqref{eq:005}}}=\int_G\overline{c_{v,v'}(g')}\,c_{h,\pi(g)^{-1}h'}(g')\,dg'\\
		&\overset{\phantom{\eqref{eq:005}}}=\scal{c_{h,\pi(g)^{-1}h'}}{c_{v,v'}}_{L^2(G)}\\
		&\overset{\eqref{eq:005}}=\frac1{d(\pi)}\scal hv_H\scal{v'}{\pi(g)^{-1}h'}_H\\
		&\overset{\phantom{\eqref{eq:005}}}=\frac1{d(\pi)}\scal hv_Hc_{v',h'}(g),\qquad g\in G.\qedhere
	\end{align*}
\end{proof}

Let $ \mathcal Z(\mathfrak g) $ denote the center of the universal enveloping algebra $ \mathcal U(\mathfrak g) $ of the complex Lie algebra $ \mathfrak g_\C=\mathfrak g\otimes_\R\C $.
We recall that a function $ \varphi:G\to\C $ is said to be:
\begin{enumerate}
	\item right $ K $-finite if $ \dim_\C\mathrm{span}_\C\left\{\varphi(\spacedcdot k):k\in K\right\}<\infty $
	\item $ \mathcal Z(\mathfrak g) $-finite if $ \varphi $ is smooth and we have $ \dim_\C\mathcal Z(\mathfrak g)\varphi<\infty $.
\end{enumerate}
It is well known that, given an integrable discrete series representation $ (\pi,H) $ of $ G $ and $ h,h'\in H_K $, the Poincar\'e series 
\[ \left(P_\Gamma c_{h,h'}\right)(g)=\sum_{\gamma\in\Gamma}c_{h,h'}(\gamma g) \]
converges absolutely and uniformly on compact subsets of $ G $ and defines a $ \mathcal Z(\mathfrak g) $-finite, right $ K $-finite function on $ \Gamma\backslash G $ that belongs to $ L^p(\Gamma\backslash G) $ for every $ p\in[1,\infty] $ (cf.\ \cite[Lemma 3.2]{zunar25b}).

The following proposition is a variation of \cite[Theorem 6.4(i)]{muic19}.

\begin{proposition}\label{lem:002}
	Let $ (\pi,H) $ be an integrable discrete series representation of $ G $, and let $ h'\in H_K\setminus\left\{0\right\} $. We define the subspace
	\[ P_\Gamma c_{H_K,h'}=\left\{P_\Gamma c_{h,h'}:h\in H_K\right\} \]
	of $ L^2(\Gamma\backslash G) $.
	Then, exactly one of the following holds:
	\begin{enumerate}[label=\textup{(\roman*)}]
		\item $ P_\Gamma c_{H_K,h'}=0 $.
		\item\label{lem:002:2} The assignment 
		\begin{equation}\label{eq:007}
			h\mapsto \frac{\norm{h'}_H}{\norm{P_\Gamma c_{h',h'}}_{L^2(\Gamma\backslash G)}} P_\Gamma c_{h,h'} 
		\end{equation}
		defines a $ (\mathfrak g,K) $-module isomorphism $ H_K\to P_\Gamma c_{H_K,h'} $ that extends uniquely to a unitary $ G $-equivalence
		\[ \Phi^\Gamma_{h'}:H\to\mathrm{Cl}_{L^2(\Gamma\backslash G)}P_\Gamma c_{H_K,h'}. \]
	\end{enumerate}	
\end{proposition}

\begin{proof}
	The assignment $ h\mapsto P_\Gamma c_{h,h'} $ is obviously $ K $-equivariant. To show that it is $ \mathfrak g $-equivariant, we note that, given $ h\in H_K $ and $ X\in\mathfrak g $, we have
	\[ \begin{aligned}
		\frac{d}{dt}\left(P_\Gamma c_{h,h'}\right)(g\exp(tX))
		&=\sum_{\gamma\in\Gamma}\frac{d}{dt}c_{h,h'}(\gamma g\exp(tX))\\
		&=\sum_{\gamma\in\Gamma}\frac{d}{ds}\bigg|_{s=0}c_{h,h'}(\gamma g\exp((t+s)X))\\
		&=\sum_{\gamma\in\Gamma}\frac{d}{ds}\bigg|_{s=0}\scal{\pi(\gamma g\exp(tX))\pi(\exp(sX))h}{h'}_H\\
		&=\sum_{\gamma\in\Gamma}\scal{\pi(\gamma g\exp(tX))\pi(X)h}{h'}_H\\
		&=\left(P_\Gamma c_{\pi(X)h,h'}\right)(g\exp(tX)),\qquad g\in G, 
	\end{aligned} \]
	where termwise differentiation in the first equality is valid since the resulting series converges absolutely and uniformly as $ t $ varies over any compact subset of $ \mathbb R $. Evaluating at $ t=0 $, we obtain
	\[ R(X)P_\Gamma c_{h,h'}=P_\Gamma c_{\pi(X)h,h'}. \]
	Thus, the assignment $ h\mapsto P_\Gamma c_{h,h'} $ defines a $ (\mathfrak g,K) $-equivariant map $ \Psi_{h'}^\Gamma:H_K\to P_\Gamma c_{H_K,h'} $.
	
	Suppose that $ P_\Gamma c_{H_K,h'}\neq0 $. Then, by the irreducibility of the $ (\mathfrak g,K) $-module $ H_K $, $ \Psi_{h'}^\Gamma $ is an isomorphism of $ (\mathfrak g,K) $-modules. Moreover, by \cite[Theorem 6.4(i)]{muic19}, $ \mathrm{Cl}_{L^2(\Gamma\backslash G)}P_\Gamma c_{H_K,h'} $ is an irreducible, closed, $ G $-invariant subspace of $ L^2(\Gamma\backslash G) $ that is unitarily equivalent to $ \pi $. By Schur's lemma \cite[\S3.3.2]{wallach88}, there exists $ c\in\R_{>0} $ such that the inifinitesimal equivalence $ c\,\Psi_{h'}^\Gamma $ of the $ G $-representations $ (\pi,H) $ and $ \mathrm{Cl}_{L^2(\Gamma\backslash G)}P_\Gamma c_{H_K,h'}  $ extends to a unitary $ G $-equivalence $ \Phi^\Gamma_{h'} $.
	By the unitarity of $ \Phi^\Gamma_{h'} $, we have
	\[ \norm{h'}_H=\norm{\Phi^\Gamma_{h'}(h')}_{L^2(\Gamma\backslash G)}=\norm{c\,\Psi^\Gamma_{h'}(h')}_{L^2(\Gamma\backslash G)}=c\norm{P_\Gamma c_{h',h'}}_{L^2(\Gamma\backslash G)}, \]
	hence $ c=\norm{P_\Gamma c_{h',h'}}_{L^2(\Gamma\backslash G)}^{-1}\norm{h'}_H $. This implies \ref{lem:002:2}.
\end{proof}

The following proof of Theorem \ref{thm:001} is based on the techniques used in \cite[proof of Theorem 2.3]{muic12} (resp., \cite[Proposition 7.1]{zunar23}) to compute the inner product of certain cuspidal automorphic forms on $ \SL_2(\R) $ (resp., $ \Sp_{2n}(\R) $). 

\begin{proof}[Proof of Theorem \ref{thm:001}]
	\ref{thm:001:1} Let $ h,h',v,v'\in H_K $. Noting that the inner product formula \eqref{eq:008} obviously holds if $ P_\Gamma c_{H_K,h'}=0 $, we suppose that $ P_\Gamma c_{H_K,h'}\neq0 $. 
	
	We have
	\begin{equation}\label{eq:009}
		\begin{aligned}
		\scal{P_\Gamma c_{h,h'}}{P_\Gamma c_{v,v'}}_{L^2(\Gamma\backslash G)}
		&\underset{\phantom{\text{Lem.\,\ref{lem:006}}}}=\int_{\Gamma\backslash G}\left(P_\Gamma c_{h,h'}\right)(g)\,\sum_{\gamma\in\Gamma}\overline{c_{v,v'}(\gamma g)}\,dg\\
		&\underset{\phantom{\text{Lem.\,\ref{lem:006}}}}=\int_{\Gamma\backslash G}\sum_{\gamma\in\Gamma}\left(P_\Gamma c_{h,h'}\right)(\gamma g)\,\overline{c_{v,v'}(\gamma g)}\,dg\\
		&\underset{\phantom{\text{Lem.\,\ref{lem:006}}}}{\overset{\eqref{eq:105}}=}\int_G\left(P_\Gamma c_{h,h'}\right)(g)\,\overline{c_{v,v'}(g)}\,dg\\
		& \underset{\text{Lem.\,\ref{lem:006}}}{\overset{\eqref{eq:002}}=}\left(R(\overline{c_{v,v'}})P_\Gamma c_{h,h'}\right)(1_G),
	\end{aligned} 
	\end{equation}
	where the right-hand side denotes the value at $ 1_G $ of the unique continuous function belonging to the equivalence class $ R(\overline{c_{v,v'}})P_\Gamma c_{h,h'}\in L^2(\Gamma\backslash G) $. 	
	
	On the other hand, by applying the $ G $-equivalence $ \Phi^\Gamma_{h'}\circ\Phi_{h'}^{-1}:c_{H,h'}\to\mathrm{Cl}_{L^2(\Gamma\backslash G)}P_\Gamma c_{H_K,h'} $ to the equality \eqref{eq:006}, we obtain
	\[ R(\overline{c_{v,v'}})P_\Gamma c_{h,h'}=\frac1{d(\pi)}\scal hv_HP_\Gamma c_{v',h'}, \]
	from which it follows that
	\begin{equation}\label{eq:010}
		\left(R(\overline{c_{v,v'}})P_\Gamma c_{h,h'}\right)(1_G)=\frac1{d(\pi)}\scal hv_H\left(P_\Gamma c_{v',h'}\right)(1_G). 
	\end{equation}
	
	The equalities \eqref{eq:009} and \eqref{eq:010} imply \eqref{eq:008}.
	
	\ref{thm:001:2} 
	Let $ h,h'\in H_K $ and $ v,v'\in V_K $. By Lemma \ref{lem:002}, we have
	\[ P_\Gamma c_{h,h'}\in L^2(\Gamma\backslash G)_{[\pi]}\qquad\text{and}\qquad  P_\Gamma c_{v,v'}\in L^2(\Gamma\backslash G)_{[\sigma]}.  \]
	Since $ L^2(\Gamma\backslash G)_{[\pi]}\perp L^2(\Gamma\backslash G)_{[\sigma]} $, this implies \ref{thm:001:2}.
\end{proof}

\section{Application to vector-valued Siegel cusp forms}\label{sec:110}

Let $ n\in\Z_{>0} $. As in \cite[Lemma 1.2]{zunar25a}, let $ (\rho,V) $ be an irreducible polynomial representation of $ \GL_n(\C) $ of highest weight $ \omega=(\omega_1,\ldots,\omega_n)\in\Z^n $, where $ \omega_1\geq\ldots\geq\omega_n>2n $. Thus, $ \rho $ is the up to equivalence unique irreducible polynomial representation of $ \GL_n(\C) $ such that there exists a unit vector $ v^{top}\in V $ with the following two properties:
\begin{enumerate}
	\item $ \rho(B)v^{top}\subseteq\C v^{top} $, where $ B $ is the subgroup of upper-triangular matrices in $ \GL_n(\C) $.
	\item $ \rho(a)v^{top}=\left(\prod_{r=1}^na_r^{\omega_r}\right)v^{top} $ for every diagonal matrix $ a=\diag(a_1,\ldots,a_n)\in\GL_n(\C) $.
\end{enumerate}
We equip $ V $ with a Hermitian inner product $ \scal\spacedcdot\spacedcdot_V $ with respect to which the restriction $ \rho\big|_{\U(n)} $ is unitary.

Let $ J_n=\begin{pmatrix}&I_n\\-I_n\end{pmatrix}\in\GL_{2n}(\C) $, and let $ i\in\C $ be the imaginary unit. The Lie group
\[ \Sp_{2n}(\R)=\left\{g\in\SL_{2n}(\R):g^\top J_ng=J_n\right\} \]
acts on the left on the Siegel upper half-space 
\[ \calH_n=\left\{z=x+iy\in M_n(\C):z^\top=z\text{ and }y>0\right\} \]
by 
\begin{equation}\label{eq:101}
	 g.z=(Az+B)(Cz+D)^{-1}
\end{equation}
and on the right on the space $ V^{\calH_n} $ of functions $ \calH_n\to V $ by
\[ \left(f\big|_\rho g\right)(z)=\rho(Cz+D)^{-1}f(g.z)  \]
for all $ g=\begin{pmatrix}A&B\\C&D\end{pmatrix}\in\Sp_{2n}(\R) $, $ z\in\calH_n $, and $ f\in V^{\calH_n} $.
The maximal compact subgroup 
\[ K=\left\{k_{A+iB}\coloneqq\begin{pmatrix}A&B\\-B&A\end{pmatrix}\in\GL_{2n}(\R):A+iB\in\U(n)\right\} \]
of $ \Sp_{2n}(\R) $ is the stabilizer of $ iI_n $ with respect to the action \eqref{eq:101}. 

Let $ \left(\pi_\rho,H_\rho\right) $ be the integrable discrete series representation of $ \Sp_{2n}(\R) $ defined on the complex Hilbert space $ H_\rho $ of holomorphic functions $ f:\calH_n\to V $ such that
\begin{equation}\label{eq:100}
	\int_{\calH_n}\norm{\rho\left(y^{\frac12}\right)f(z)}^2_V\,d\mathsf v(z)<\infty, 
\end{equation}
where $ d\mathsf v(z)=\det y^{-n-1}\,\prod_{1\leq r\leq s\leq n}dx_{r,s}\,dy_{r,s} $.
The Hilbert space norm on $ H_\rho $ is given by the square root of \eqref{eq:100}, and the action of $ \Sp_{2n}(\R) $ on $ H_\rho $ is given by
\[ \pi_\rho(g)f=f\big|_\rho g^{-1},\quad g\in\Sp_{2n}(\R),\ f\in H_\rho. \] 
The $ (\mathfrak g,K) $-module $ (H_\rho)_K $ of $ K $-finite vectors in $ H_\rho $ is spanned by the functions $ f_{\mu,v}:\calH_n\to V $,
\begin{equation}\label{eq:109}
	f_{\mu,v}(z)=\mu\left((z-iI_n)(z+iI_n)^{-1}\right)\rho\left(\frac1{2i}(z+iI_n)\right)^{-1}v,  
\end{equation}
where $ v $ goes over $ V $ and $ \mu $ goes over $ \C[X_{r,s}:1\leq r,s\leq n] $, the ring of polynomials with complex coefficients in the $ n^2 $ variables $ X_{r,s} $ with $ r,s\in\left\{1,2,\ldots,n\right\} $ \cite[(15)]{zunar25a}. 

Given a complex Hilbert space $ H $, let us denote its dual Hilbert space by $ H^* $; thus, $ H^* $ is the space of linear functionals $ h^*=\scal\spacedcdot h_H:H\to\C $, where $ h\in H $.
The contragredient representation  of $ \pi_\rho $ is the integrable discrete series representation $ (\pi_\rho^*,H_\rho^*) $ of $ \Sp_{2n}(\R) $ given by
\[ \pi_\rho^*(g)f^*=\left(\pi_\rho(g)f\right)^*,\quad g\in\Sp_{2n}(\R),\ f\in H_\rho. \]
Obviously, we have
\[ (H_\rho^*)_K=\mathrm{span}_\C\left\{f_{\mu,v}^*:\mu\in\C[X_{r,s}:1\leq r,s\leq n],\ v\in V\right\}. \]

Given a function $ f:\calH_n\to V $, we define a function $ F_f:\Sp_{2n}(\R)\to V $,
\begin{equation}\label{eq:103}
	 F_f(g)=\left(f\big|_\rho g\right)(iI_n). 
\end{equation}
By \cite[Proposition 6.4(ii)]{zunar25a}, given $ f\in(H_\rho)_K $ and $ v\in V $, the $ K $-finite matrix coefficient  $ c_{f_{1,v}^*,f^*} $ of the representation $ \pi_\rho^* $ is given by
\begin{equation}\label{eq:054}
	c_{f_{1,v}^*,f^*}=C_\rho\, v^*F_f, 
\end{equation}
where 
\[ C_\rho=\norm{f_{1,v^{top}}}_{H_\rho}^2\in\R_{>0}. \]

\begin{lemma}
	Let $ v\in V $. Then, we have
	\begin{equation}\label{eq:107}
		\norm{f_{1,v}^*}_{H_\rho^*}^2=C_\rho\norm v_V^2. 
	\end{equation}
\end{lemma}

\begin{proof}
	We have
	\begin{align*}
		\norm{f_{1,v}^*}_{H_\rho^*}^2
		&=c_{f_{1,v}^*,f_{1,v}^*}(1_{\Sp_{2n}(\R)})
		\overset{\eqref{eq:054}}=C_\rho\scal{F_{f_{1,v}}(1_{\Sp_{2n}(\R)})}{v}_V\\
		&\overset{\eqref{eq:103}}=C_\rho\scal{f_{1,v}(iI_n)}v_V
		\overset{\eqref{eq:109}}=C_\rho\scal vv_V
		=C_\rho\norm v_V^2.\qedhere
	\end{align*} 
\end{proof}

Let $ \Gamma $ be a discrete subgroup of $ \Sp_{2n}(\R) $ that is commensurable with $ \Sp_{2n}(\Z) $. The space $ S_\rho(\Gamma) $ of Siegel cusp forms of weight $ \rho $ for $ \Gamma $ is the finite-dimensional space of holomorphic functions $ \phi:\calH_n\to V $ with the following two properties (cf.\ \cite[Lemma 4.1]{zunar25a}):
\begin{enumerate}
	\item $ \phi\big|_\rho\gamma=\phi $ for all $ \gamma\in\Gamma $.
	\item $ \sup_{z\in\calH_n}\norm{\rho\left(y^{\frac12}\right)\phi(z)}_V<\infty $.
\end{enumerate}
We equip $ S_\rho(\Gamma) $ with the Petersson inner product
\[ \scal{\phi_1}{\phi_2}_{S_\rho(\Gamma)}=\frac1{\abs{\Gamma\cap\left\{\pm I_{2n}\right\}}}\int_{\Gamma\backslash\calH_n}\scal{\rho(y)\,\phi_1(z)}{\phi_2(z)}_V\,d\mathsf v(z). \]

\begin{proposition}\label{prop:105}
	\begin{enumerate}[label=\textup{(\roman*)}]
		\item\label{prop:105:1} For every $ f\in(H_\rho)_K $, the Poincar\'e series
		\[ P_{\Gamma,\rho}f=\sum_{\gamma\in\Gamma}f\big|_\rho\gamma \]
		converges absolutely and uniformly on compact subsets of $ \calH_n $.
		\item\label{prop:105:2} We have
		\[ S_\rho(\Gamma)=\left\{P_{\Gamma,\rho}f:f\in(H_\rho)_K\right\}. \]
		\item\label{prop:105:3} The assignment $ \phi\mapsto F_\phi $ defines a unitary isomorphism $ \Phi_{\rho,\Gamma} $ from $ S_\rho(\Gamma) $ to a finite-dimensional subspace $ L_\rho(\Gamma) $ of the Hilbert space $ L^2(\Gamma\backslash\Sp_{2n}(\R),V) $.
		\item\label{prop:105:4} For every $ f\in(H_\rho)_K $, the Poincar\'e series 
		\[ P_\Gamma F_f=\sum_{\gamma\in\Gamma}F_f(\gamma\spacedcdot) \]
		converges absolutely and uniformly on compact subsets of $ \Sp_{2n}(\R) $, and we have
		\begin{equation}\label{eq:106}
			 \Phi_{\rho,\Gamma}(P_{\Gamma,\rho}f)=F_{P_{\Gamma,\rho}f}=P_\Gamma F_f. 
		\end{equation}
	\end{enumerate}
\end{proposition}

\begin{proof}
	The claims \ref{prop:105:1} and \ref{prop:105:2} follow from \cite[Theorem 6.7]{zunar25a}, \ref{prop:105:3} follows from  \cite[Lemma 4.4]{zunar25a}, and \ref{prop:105:4} follows from \ref{prop:105:1} and \cite[Lemma 6.6]{zunar25a}.
\end{proof}

With the above results and the inner product formula \eqref{eq:008} at hand, we are ready to give a re\-pre\-sen\-ta\-tion-the\-o\-re\-tic proof of Corollary \ref{cor:111}, which we restate here for the reader's convenience as Corollary \ref{cor:112}.

\begin{corollary}\label{cor:112}
	Let $ \phi\in S_\rho(\Gamma) $ and $ v\in V $. Then, we have
	\[ \scal{\phi}{P_{\Gamma,\rho}f_{1,v}}_{S_\rho(\Gamma)}=\frac{\dim_\C V}{d(\pi_\rho)}\scal{\phi(iI_n)}v_{V}. \]
\end{corollary}

\begin{proof}
	Let us fix an orthonormal basis $ (e_j)_{j=1}^{\dim_\C V} $ for $ V $. We note that for all $ \varphi_1,\varphi_2\in L^2(\Gamma\backslash \Sp_{2n}(\R),V) $, we have
	\begin{equation}\label{eq:113}
			\scal{\varphi_1}{\varphi_2}_{L^2(\Gamma\backslash \Sp_{2n}(\R),V)}
			=\sum_{j=1}^{\dim_\C V}\scal{e_j^*\varphi_1}{e_j^*\varphi_2}_{L^2(\Gamma\backslash \Sp_{2n}(\R))}.
	\end{equation}
	
	By Proposition \ref{prop:105}\ref{prop:105:2}, we have
	\begin{equation}\label{eq:108}
		\phi=P_{\Gamma,\rho}f
	\end{equation}
	for some $ f\in(H_\rho)_K $.  The Petersson inner product
	\[ \scal{\phi}{P_{\Gamma,\rho}f_{1,v}}_{S_\rho(\Gamma)}=\scal{P_{\Gamma,\rho}f}{P_{\Gamma,\rho}f_{1,v}}_{S_\rho(\Gamma)} \]
	by Proposition \ref{prop:105}\ref{prop:105:3} and \eqref{eq:106} equals
	\[ \begin{aligned}
		\scal{P_{\Gamma}F_f}{P_{\Gamma}F_{f_{1,v}}}_{L^2(\Gamma\backslash\Sp_{2n}(\R),V)}
		&\overset{\eqref{eq:113}}=\sum_{j=1}^{\dim_\C V}\scal{e_j^*P_\Gamma F_f}{e_j^*P_{\Gamma}F_{f_{1,v}}}_{L^2(\Gamma\backslash \Sp_{2n}(\R))}\\
		&\overset{\phantom{\eqref{eq:113}}}=\sum_{j=1}^{\dim_\C V}\scal{P_\Gamma e_j^*F_f}{P_{\Gamma}e_j^*F_{f_{1,v}}}_{L^2(\Gamma\backslash \Sp_{2n}(\R))},
		\end{aligned} \]
	which by the formula \eqref{eq:054} for matrix coefficients of $ \pi_\rho^* $ equals 
	\begin{equation}\label{eq:110}
		C_\rho^{-2}\sum_{j=1}^{\dim_\C V}\scal{P_\Gamma c_{f_{1,e_j}^*,f^*}}{P_{\Gamma}c_{f_{1,e_j}^*,f_{1,v}^*}}_{L^2(\Gamma\backslash \Sp_{2n}(\R))}. 
	\end{equation}
	Applying the inner product formula \eqref{eq:008}, we rewrite \eqref{eq:110} as
	\[ C_\rho^{-2}\sum_{j=1}^{\dim_\C V}\frac1{d(\pi_\rho^*)}\norm{f_{1,e_j}^*}^2_{H_\rho^*}\left(P_\Gamma c_{f_{1,v}^*,f^*}\right)\left(1_{\Sp_{2n}(\R)}\right) \]
	and then, noting that $ d(\pi_\rho^*)=d(\pi_\rho) $ and applying \eqref{eq:107} and, for the second time, the formula \eqref{eq:054} for matrix coefficients of $ \pi_\rho^* $, as
	\[ \begin{aligned}
		\frac{\dim_\C V}{d(\pi_\rho)}&\left(P_\Gamma v^*F_f\right)\left(1_{\Sp_{2n}(\R)}\right)
		\overset{\eqref{eq:106}}=\frac{\dim_\C V}{d(\pi_\rho)}\left(v^*F_{P_{\Gamma,\rho}f}\right)\left(1_{\Sp_{2n}(\R)}\right)\\
		&\overset{\eqref{eq:108}}=\frac{\dim_\C V}{d(\pi_\rho)}\scal{F_\phi\left(1_{\Sp_{2n}(\R)}\right)}v_V
		\overset{\eqref{eq:103}}=\frac{\dim_\C V}{d(\pi_\rho)}\scal{\phi(iI_n)}v_V,
	\end{aligned} \]
	thus finishing the proof of the corollary.
\end{proof}

\bibliographystyle{amsplain}

\begin{thebibliography}{999999}
	
	\bibitem[Dix77]{dixmier77} Dixmier, J.: \textit{$ C^* $-algebras}, North-Holland Math.\ Library \textbf{15}, North-Holland Publishing Co., Amsterdam-New York-Oxford (1977)
	
	\bibitem[God57]{godement57} Godement, R.: S\'erie de Poincar\'e et spitzenformen. \textit{S\'eminaire Henri Cartan}, tome 10, no 1, exp. no 10, 1--38 (1957-1958)
	
	\bibitem[Kna02]{knapp02} Knapp, A.\ W.: \textit{Lie groups beyond an introduction},	Progr.\ Math.\ \textbf{140}, Birkh\"auser Boston, Inc., Boston, MA (2002)
	
	\bibitem[Kna86]{knapp86} Knapp, A.\ W.:	\textit{Representation theory of semisimple groups. An overview based on examples}. Princeton Math.\ Ser.\ \textbf{36}, Princeton University Press, Princeton, NJ (1986)
	
	
	\bibitem[Mui12]{muic12} Mui\'c, G.:	On the inner product of certain automorphic forms and applications.	\textit{J.\ Lie Theory} \textbf{22}(4), 1091--1107 (2012)
	
	\bibitem[Mui16]{muic16} Mui\'c, G.:	Fourier coefficients of automorphic forms and integrable discrete series. \textit{J.\ Funct.\ Anal.\ }\textbf{270}(10), 3639--3674 (2016)
	
	\bibitem[Mui19]{muic19} Mui\'c, G.: Smooth cuspidal automorphic forms and integrable discrete series. \textit{Math.\ Z.\ }\textbf{292}(3--4), 895--922 (2019)
	
	\bibitem[Wall88]{wallach88} Wallach, N.~R.:	\textit{Real reductive groups. I.} Pure Appl.\ Math.\ \textbf{132},	Academic Press, Inc., Boston, MA (1988)
	
	\bibitem[War72]{warner72} Warner, G.: \textit{Harmonic analysis on semi-simple Lie groups. I}, Die Grundlehren der mathematischen Wissenschaften \textbf{188},	Springer-Verlag, New York-Heidelberg (1972) 
	
	\bibitem[\v Zun23]{zunar23} \v Zunar, S.: On a family of Siegel Poincar\'e series. \textit{Int.\ J.\ Number Theory} \textbf{19}(9), 2215--2239 (2023)
	
	\bibitem[\v Zun25a]{zunar25a} \v Zunar, S.: Construction and non-vanishing of a family of vector-valued Siegel Poincaré series. \textit{J.\ Number Theory }\textbf{268}, 95--123 (2025)
	
	\bibitem[\v Zun25b]{zunar25b} \v Zunar, S.: \textit{On the non-vanishing of Poincar\'e series on irreducible bounded symmetric domains}, preprint, \href{https://arxiv.org/abs/2501.06876v1}{arXiv:2501.06876v1}
	
	
	
\end{thebibliography}
\linespread{.96}

\end{document}